\documentclass[12pt]{article}
\usepackage{amsmath}
\usepackage{amssymb}
\usepackage{amsthm}
\usepackage{graphicx}
\usepackage{multirow}
\usepackage{colortbl}
\usepackage{psfrag}

\usepackage[english]{babel}
\usepackage[ansinew]{inputenc}

\theoremstyle{plain} 
\newtheorem{theorem}{Theorem}
\newtheorem{lemma}{Lemma}
\newtheorem{corollary}{Corollary}

\theoremstyle{definition} 

\newtheorem{remark}{Remark}

\newcommand{\be}{\begin{eqnarray}}
\newcommand{\ee}{\end{eqnarray}}

\newcommand{\bn}{\begin{eqnarray*}}       
\newcommand{\en}{\end{eqnarray*}}
\newcommand{\beq}{\begin{equation}}
\newcommand{\eeq}{\end{equation}}

\newcommand{\R}{\mathbb{R}}
\newcommand{\N}{\mathbb{N}}

\newcommand{\eps}{\varepsilon}

\DeclareMathOperator{\conv}{conv}
\DeclareMathOperator{\dive}{div}

\DeclareMathOperator{\dist}{dist}

\newcommand{\ba}{\begin{array}}
\newcommand{\ea}{\end{array}}

\newcommand{\pd}[2]{\frac{\partial #1}{\partial #2}}
\newcommand{\pdt}[2]{\tfrac{\partial #1}{\partial #2}}

\newcommand{\bpm}{\begin{pmatrix}}
\newcommand{\epm}{\end{pmatrix}}

\begin{document}

\title{A computer-assisted existence proof for Emden's equation on an unbounded L-shaped domain}
\author{Filomena Pacella, Michael Plum and Dagmar Rütters}
\date{}
\maketitle
\begin{abstract}
We prove existence, non-degeneracy, and exponential decay at infinity of a non-trivial solution to Emden's equation $-\Delta u = | u |^3$ on an unbounded $L$-shaped domain, subject to Dirichlet boundary conditions. Besides the direct value of this result, we also regard this solution as a building block for solutions on expanding bounded domains with corners, to be established in future work. Our proof makes heavy use of computer assistance: Starting from a numerical approximate solution, we use a fixed-point argument to prove existence of a near-by exact solution. The eigenvalue bounds established in the course of this proof also imply non-degeneracy of the solution.
\end{abstract}

\section{Introduction}
In this paper we are concerned with the existence of non-trivial solutions of the problem
\be\label{In1}
\left\{ \begin{array}{rcll}
-\Delta u &=& |u|^3 & {\rm in~} \Omega\\
u &=& 0 & {\rm on~} \partial\Omega \end{array} \right.
\ee
in a planar $L$-shaped unbounded domain
\bn
\Omega = ((-1,\infty) \times (0,1)) \cup ((-1,0) \times (-\infty,1)) \subset \R^2.
\en
It is obviously of great importance, both theoretically and for applications, not only to find a solution but also to detect its shape and other qualitative properties. In particular, for domains with corners, it is very interesting to find solutions which are localized at the corners as, often, one can guess from applications or from energy considerations.

Here, by computer-assistance, we prove the existence of such a solution which is close to an approximate numerically computed solution with the desired shape. The method and the precise result will be stated in the next section.

Moreover, solutions in unbounded domains can be used to find solutions in bounded domains by perturbative methods. Roughly speaking in some nonlinear parameter dependent elliptic problems in bounded domains which, as the parameter goes to a limit value, ``tend'' to a similar problem in an unbounded domain, one can use the solution in the unbounded domain to construct solutions of the original problem. This is usually done by ``gluing'' together several ``copies'' of suitable truncations of the solution in the unbounded domain. This technique usually requires two main properties of the solution in the unbounded domain:
\be\label{InI}
&\hspace*{-8,7cm}{\rm i)} & {\rm ~it~ must~ be~ nondegenerate,}\\ \nonumber
&\hspace*{-8,7cm}{\rm ii)} & {\rm ~it~ should~ have~ a~ suitably ~ fast ~ decay ~at~ infinity.}
\ee
Some typical cases, widely analyzed in the literature, are so-called singularly-perturbed problems or problems with critical  nonlinearities. In a similar way a solution of \eqref{In1} can be used to find solutions in expanding bounded tubular domains.

Let us be more precise. Let $M$ be a compact $k$-dimensional piecewise smooth submanifold of $\R^N$ without boundary, $1 \le k \le N-1, ~ N \ge 2$. For $R>0$ define the expanded manifold $M_R = \{ Rx,x \in M\}$ and denote by $\Omega_R$ its open tubular neighborhood of radius 1. Then consider the problem 
\be\label{In3}
\left\{ \begin{array}{rcll}
-\Delta u &=& f(u)& {\rm in~} \Omega_R\\
u &=& 0 & {\rm on~} \partial\Omega_R \end{array} \right.
\ee
for some nonlinear $C^1$-function $f$. The equation \eqref{In3} appears e.g. in nonlinear optics and models standing waves in optical waveguides. The most interesting variant for applications which exploits the nonlinear properties of the material is the {\it self-focusing} case where $f(u)/u \to \infty$ as $|u| \to \infty$. A typical example is given by $f(u) = u^3$, modelling Kerr's effect. For more information on the physical background see for example \cite{Su-Su}.

When $k=1$ then $\Omega_R$ is a tubular guide, i.e. an optical fiber and the flat case, $N=2$, is an interesting one. When $k\ge 2$ and $M$ is regular and symmetric, in particular $\Omega_R$ is an expanding annuli, some existence results have been found using variational methods (\cite{CW}, \cite{BCGP}, \cite{YYL}, \cite{Su}). 

The first paper which deals with the nonsymmetric case, but always assuming $M$ to be regular, is \cite{DY}, where solutions are found in the form
\be\label{In4}
u_R = \sum\limits^{m}_{i=1} U_{X_{R,i}, R} + o (1)
\ee
in $H^1 (\R^N)$ as $R \to \infty$, where $U_{X_{R,i}, R}$ are solutions of the ``local'' limit problem in an open unbounded cylinder centered at suitable points $X_{R,i},~ i = 1, \dots, m$ of the domain $\Omega_R$. Solutions of the form \eqref{In4} are usually called ``multibump'' solutions. This result was improved in \cite{Ack-Cla-Pa1} for positive solutions and extended then to sign-changing solutions in \cite{Ack-Cla-Pa2} in the case $k=1$, but always starting from regular manifolds.

If the original $1$-dimensional manifold has some ``corners'' one would expect the existence of similar ``multibump'' solutions, but with (some) bumps localized at the corners. In order to find such a type of solutions the main thing is to have a solution of a limit problem in an unbounded domain with a corner which is precisely localized at a corner. More precisely, if the $1$-dimensional manifold $M$ is a piecewise regular manifold in the plane given by the union of a finite number of smooth curves which intersect orthogonally one could use the solution we find in the $L$-shaped domain $\Omega$ to construct a multibump solution of type \eqref{In4} with the bumps (positive or negative) at the corners.

It is important at this point to stress what we have already mentioned in \eqref{InI}, i.e. that in order to use the solution $u$ in the unbounded $L$-shaped domain as building block for solutions in expanding bounded domains with corners, we need the two main properties i) and ii) (see \eqref{InI}). In other words the solution $u$ must be nondegenerate and decay exponentially at infinity.

In this paper we also prove both properties of the solution found by a computer-assisted proof and we compute its Morse index (which is one). We believe that all these results are interesting in themselves.

Finally, let us explain briefly the numerical and analytical techniques that we use for our computer-assisted proof. We start with an approximate solution $\omega \in H_0^1 (\Omega)$ to problem \eqref{In1} which we compute by a Newton iteration combined with a finite element method; in order to overcome the singularity problems at the re-entrant corner $(0,0)$ of $\Omega$, we also involve the associated singularity function into the approximation process. In this first approximative step, there is no need for any mathematical rigor.

Now we put up a boundary value problem for the error $v = u-w$, which we re-write as a fixed-point equation involving the residual $-\Delta \omega - |\omega|^3$ of the approximate solution, and the inverse of the linearization $L_{\omega} = - \Delta - 3 | \omega | \omega$. Now Banach's Fixed-Point Theorem gives the existence of a solution to problem \eqref{In1} in some ``close'' $H_0^1(\Omega)$-neighborhood of $\omega$, provided that the residual is sufficiently small (measured in $H^{-1} (\Omega)$) and $\| L_{\omega}^{-1} \|_{H^{-1} (\Omega) \to H_0^1 (\Omega)}$ is ``moderate''; the precise conditions are formulated in Theorem 1.

For computing a rigorous bound to the residual we use an additional $H({\rm div},\Omega)$-approximation to $\nabla \omega$ (see Section 4) and rigorous bounds to some integrals. A norm bound for $L_{\omega}^{-1}$ is obtained via spectral bounds for $L_{\omega}$: The essential spectrum can be bounded by simple Rayleigh quotient estimates, and isolated eigenvalues below the essential spectrum are enclosed by variational methods and additional computer-assisted means, supplemented by a homotopy method for obtaining some necessary spectral a priori information.

By related computer-assisted techniques we have been able to prove existence, multiplicity, and also uniqueness statements for various kinds of boundary and eigenvalue problems; see e.g. \cite{BehMerPlWie,Breuerbeam,HoangPlWie,McKenPacPlR,MR3203775,PlumDMV}.

The paper is organized as follows. In Section 2 we formulate the basic theorem (Theorem 1) for our computer-assisted proof. Section 3 contains a description of the numerical methods used to obtain an approximate solution $\omega$, and Section 4 the essential estimates for obtaining a residual bound. Section 5 is devoted to the spectral estimates for $L_{\omega}$ needed to bound $\| L_{\omega}^{-1} \|_{H^{-1} (\Omega) \to H_0^1 (\Omega)}$. Section 6 contains some more numerical details and our computational results proving the existence of a solution to problem \eqref{In1} together with a close $H_0^1 (\Omega)$-error bound. Based on the spectral information obtained as a ``side product'' of our computer-assisted proof, we prove nondegeneracy of the solution in Section 7. Finally, Section 8 contains a proof of exponential decay of the solution. 


\section{Existence and Enclosure Theorem}

Let $\Omega=\left ((-1,\infty)\times(0,1)\right)\cup\left ((-1,0)\times(-\infty,1)\right )$, i.e. $\Omega$ is an
unbounded L-shaped
domain. Let $H_0^1(\Omega)$ be endowed with the inner product $\langle
u,v\rangle_{H_0^1}:=\langle\nabla u,\nabla v\rangle_{L^2} +\langle u,v\rangle_{L^2}$, and let $H^{-1}(\Omega)$
denote the dual space of $H_0^1(\Omega)$, equipped with the usual operator sup-norm. We consider the problem
\beq
\label{bvp}
\left \{\begin{array}{rcll} -\Delta u & = & |u|^3 & \text{in }\Omega\\ u & = & 0 & \text{on }\partial\Omega.\end{array}
\right .
\eeq
Our goal is to prove existence of a non-trivial solution by computer-assistance.

We assume that $\omega\in H_0^1(\Omega)$ is an approximate solution to \eqref{bvp} (computed numerically) and that constants $\delta$ and $K$
are known  such that
\begin{itemize}
\item[(a)] $\delta$ bounds the defect (residual) of the approximate solution in the $H^{-1}$-norm, i.e.
\be
\label{deltacond}
\|-\Delta \omega-|\omega|^3\|_{H^{-1}}\leq\delta,
\ee
\item[(b)] $K$ bounds the inverse of the linearization at $\omega$, i.e.
\be
\label{Kcond}
\|v\|_{H_0^1}\leq K\left \|L_\omega[v]\right \|_{H^{-1}}\quad\text{for all }v\in H_0^1(\Omega),
\ee
where $L_\omega:H_0^1(\Omega)\to H^{-1}(\Omega),\ L_\omega[v]=-\Delta v-3|\omega|\omega v$ denotes the linearized
operator.
\end{itemize}
Note that condition \eqref{Kcond} implies that $L_\omega$ is one-to-one. We will also need that $L_\omega$
is
onto. For proving this we use the canonical isometric isomorphism $\Phi:H_0^1(\Omega)\to H^{-1}(\Omega)$ given by
\be
\label{Phidef}
(\Phi[u])(v):=\langle u,v\rangle_{H_0^1}\quad(u,v\in H_0^1(\Omega))
\ee
and show that 
\begin{itemize}
\item[(i)] $\left (\Phi^{-1}L_\omega\right)(H_0^1(\Omega))$ is dense in $H^1_0(\Omega)$ (implying that
$L_\omega(H_0^1(\Omega))\subset H^{-1}(\Omega)$ is dense),
\item[(ii)] $L_\omega(H_0^1(\Omega))\subset H^{-1}(\Omega)$ is closed.
\end{itemize}
For proving (i) we first note that $\Phi^{-1}L_\omega:H_0^1(\Omega)\to H_0^1(\Omega)$ is symmetric w.r.t. $\langle
\cdot,\cdot\rangle_{H_0^1}$:
\begin{align}
\langle \Phi^{-1}L_\omega[u],v\rangle_{H_0^1} & \stackrel{\eqref{Phidef}}{= }
\left(L_\omega[u]\right)[v] =\int_\Omega\nabla u\cdot\nabla v-3|\omega|\omega uv\,dx.
\label{PhiLsymm}
\end{align}
Let now $u\in H_0^1(\Omega)$ be an element of the orthogonal complement of $(\Phi^{-1}L_\omega)(H_0^1(\Omega))$, i.e. we
have $$0=\langle u,\Phi^{-1}L_\omega[v]\rangle_{H_0^1}\stackrel{\text{symmetry}}{=}
\langle \Phi^{-1}L_\omega[u],v\rangle_{H_0^1}\quad\text{for all }v\in H_0^1(\Omega).$$
Therefore $\Phi^{-1}L_\omega[u]=0$, which implies $L_\omega[u]=0$ and, since
$L_\omega$ is one-to-one, $u=0$. Thus (i) follows.

To prove (ii), let $\left (L_\omega[u_n]\right)_{n\in\N}$ be a sequence in $L_\omega(H_0^1(\Omega))$ converging to some
$r \in H^{-1}(\Omega)$. Condition \eqref{Kcond} shows that $(u_n)_{n\in\N}$ is a Cauchy
sequence in $H_0^1(\Omega)$ and thus converges to some $u\in H_0^1(\Omega)$. Since $L_\omega$ is bounded we obtain
$L_\omega[u_n]\to L_\omega[u]\ \ (n\to\infty)$, which gives $r =L_\omega[u]\in L_\omega(H_0^1(\Omega))$ and therefore the closedness of
$L_\omega(H_0^1(\Omega))$ in $H^{-1}(\Omega)$.

\begin{theorem}
\label{maintheo}
Let $\omega\in H_0^1(\Omega)$ be an approximate solution to \eqref{bvp}, and $\delta$ and
$K$ constants such that \eqref{deltacond} and \eqref{Kcond} are satisfied. Let moreover $C_4>0$ be an embedding constant
for the embedding $H_0^1(\Omega)\hookrightarrow L^4(\Omega)$, and $\gamma:=3C_4^3$. \\Finally suppose that there exists
some $\alpha>0$ such that 
\beq
\label{alphacond1}
\delta\leq \frac{\alpha}{K}-\gamma\alpha^2\left (\|\omega\|_{L^4}+\tfrac{1}{3}C_4\alpha\right)
\eeq
and
\beq
\label{alphacond2}
2K\gamma\alpha\left (\|\omega\|_{L^4}+\tfrac 12C_4\alpha\right) <1.
\eeq
Then there exists a solution $u\in H_0^1(\Omega)$ to problem \eqref{bvp} such that
\beq
\label{encl}
\|u-\omega\|_{H_0^1}\leq \alpha.
\eeq
In particular, $u$ is non-trivial if $\alpha<\|\omega\|_{H_0^1}$.
\end{theorem}

\begin{remark}
\begin{itemize}
\item[(a)] Let $\psi(\alpha)$ denote
the right-hand-side of \eqref{alphacond1}, which obviously attains a positive maximum on $[0,\infty)$. Thus
the existence of some $\alpha>0$ satisfying \eqref{alphacond1} is equivalent to 
\beq\label{dcond2}
\delta\leq\max\limits_{\alpha\in[0,\infty)}\psi(\alpha),
\eeq
which due to \eqref{deltacond} will be satisfied if the approximate solution $\omega$ is
computed with sufficiently high accuracy. Furthermore, a small defect bound $\delta$ will imply a small error bound $\alpha$ if $K$ is not too large.
\item[(b)] As shown in \cite{McKenPacPlR}, \eqref{alphacond1} implies \eqref{alphacond2} if $\delta$ satisfies \eqref{dcond2} with a strict inequality 
and $\alpha$ is chosen appropriately. 
\end{itemize}
\end{remark}

For the proof of Theorem \ref{maintheo} see \cite{McKenPacPlR}. It is based on Banach's Fixed Point Theorem and uses the following Lemma (see \cite[Lemma 3.1 and
3.2]{McKenPacPlR}), which in addition will be useful in a later section of this paper.
\begin{lemma}
\label{l1}
Let $p_1,p_2,p_3,p_4\in[2,\infty)$ such that $\frac 1{p_1}+\frac 1{p_2}+\frac 1{p_3}+\frac 1{p_4}=1$ and $C_{p_i}>0$ 
an embedding constant for the embedding $H_0^1(\Omega)\hookrightarrow L^{p_i}(\Omega)$, $i=1,\ldots,4$. Then (i) and (ii) hold true:
\begin{itemize}
\item[\rm{(i)}] For all $u,\tilde u,v\in H_0^1(\Omega)$:
$$\left \|\left [|u|u-|\tilde u|\tilde u\right ]v\right \|_{H^{-1}}\leq C_{p_3}C_{p_4}\left(\|u\|_{L^{p_1}}+\|\tilde
u\|_{L^{p_1}}\right )\|u-\tilde  u\|_{L^{p_2}}\|v\|_{H_0^1}.$$
\item[\rm{(ii)}]
Let $u,\tilde u \in H_0^1(\Omega)$, and let $L_u$ and $L_{\tilde{u}}$ denote the linearizations at $u$ and $\tilde{u}$, respectively. Suppose that for some $K>0$
$$\|v\|_{H_0^1}\leq K\|L_{\tilde u}[v]\|_{H^{-1}}\quad\text{for all }\ v\in H_0^1(\Omega)$$
and 
\beq\label{kappacond}
\kappa:=3C_{p_3}C_{p_4} K\left (\|u\|_{L^{p_1}}+\|\tilde u\|_{L^{p_1}}\right )\|u-\tilde u\|_{L^{p_2}}<1.
\eeq
Then,
$$\|v\|_{H_0^1}\leq\frac{K}{1-\kappa}\left \|L_u[v]\right\|_{H^{-1}}\quad\text{for all }\ v\in H_0^1(\Omega).$$
\end{itemize}
\end{lemma}
\begin{remark}
Note that, in \eqref{kappacond}, $\| u - \tilde{u}\|_{L^{p_2}}$ can be replaced by $C_{p_2} \| u - \tilde{u}\|_{H_0^1}$, which for the particular choice $p_1=p_2=p_3=p_4=4$ amounts to the condition
$$\kappa := \gamma K\left (\|u\|_{L^{4}}+\|\tilde u\|_{L^{4}}\right )\|u-\tilde u\|_{H_0^1}<1.$$
\end{remark}

\section{Computation of an approximate solution}

\label{approxsol}
Let $T>1$ and define $\Omega_0:=\Omega\cap(-T,T)^2$. Then $\Omega_0$ is bounded and contains
the corner-part of $\Omega$. Let $\omega_0\in H_0^1(\Omega_0)$ be an approximate solution of
\beq
\label{cbvp}
\left \{\begin{array}{rcll}
-\Delta u & = & |u|^3 & \text{in }\Omega_0 \\ u & = & 0 & \text{on }\partial\Omega_0.\end{array}\right .
\eeq
Then
\beq
\label{omdef}
\omega=\left \{\begin{array}{ccc} \omega_0 & \text{in} & \Omega_0\\
0 & \text{in} & \Omega\backslash\Omega_0\end{array}\right .
\eeq
is in $H_0^1(\Omega)$ and turns out to be a good approximate solution of \eqref{bvp} if $T$ is chosen large enough and
$\omega_0$ is sufficiently accurate. Indeed our numerical results  show that $T=3$ is sufficient. 

An approximate solution of \eqref{cbvp} can be computed using finite elements and a Newton iteration. As an initial guess,
an appropriate multiple of the first eigenfunction of $-\Delta$ in $(-1,0)\times(0,1)$ with homogeneous Dirichlet boundary
conditions and extended by zero to $\Omega_0\backslash((-1,0)\times(0,1))$ can be used. However, due to
the re-entrant corner at $0$, approximations obtained with finite elements alone do not yield a sufficiently small
defect. So in addition we use a corner singular function, which allows us to write the exact solution of
\eqref{cbvp}
as a sum of a singular part and a regular part in $H^2(\Omega_0)$ (see \cite{GrisvNonsm} and \cite{GrisvSing}): 
We introduce polar coordinates $(r,\varphi)$, where $r=|x|$ and $\varphi$ ranges between $0$ and
$\theta:=\frac{3\pi}2$, with $\varphi=0$ for $x>0,y=0$ and $\varphi=\theta$ for $x=0,y<0$. On $\overline\Omega_0$ we
define
\beq
\label{gammadef}
\gamma(r,\varphi):=r^{\frac 23}\sin(\tfrac 23\varphi).
\eeq
Obviously, $\gamma(r,0)=\gamma(r,\theta)=0$ and one can easily check that $\Delta\gamma =0$ in $\Omega_0$. We choose
some fixed function $\lambda\in H^2(\Omega_0)\cap C^1(\overline\Omega_0)$ which vanishes on the part of $\partial\Omega_0$
where $\gamma$ does not vanish and satisfies $\lambda(0)=1,\ \nabla\lambda(0)=0$. With $w:=\lambda\gamma\in H_0^1(\Omega_0)$,
a solution $u\in H_0^1(\Omega_0)$ to \eqref{cbvp} can then be written as
\beq
\label{singsplit}
u = aw+v,
\eeq

where $v\in H^2(\Omega_0)\cap H_0^1(\Omega_0)$ is the regular part and $a\in\R$ is the so-called
stress-intensity-factor.
Using the dual singular function $\Gamma(r,\varphi):=r^{-\tfrac 23}\sin(\tfrac 23\varphi)$ we can represent $a$
by means of the solution $u$, i.e. we have
\beq
\label{aeq}
a=\frac 1\pi\left (\int_{\Omega_0} (\Lambda\Gamma)|u|^3\,+\Delta(\Lambda\Gamma)u\,dx\right),
\eeq
where $\Lambda\in H^2(\Omega_0)\cap C^1(\overline\Omega_0)$ is a cutoff function with similar properties as $\lambda$
and such that $\Delta(\Lambda\Gamma)\in L^2(\Omega_0)$. A suitable choice for $\lambda$ and $\Lambda$ is e.g. given by $\lambda(x,y)=\Lambda(x,y)=(1-x^2)^2(1-y^2)^2\chi_{(-1,1)^2}(x,y)$ for $(x,y)\in\overline\Omega_0$.
 
Clearly, a computation of the exact stress-intensity-factor by \eqref{aeq} is impossible since the exact
solution $u$ is unknown. For our purpose - the improvement of the approximate solution - it is however sufficient to know only an
approximation of $a$. So let a finite element function $\tilde\omega_0$ (computed without separate singular part) be an
approximate solution of \eqref{cbvp}.
Inserting $\tilde\omega_0$ into \eqref{aeq} yields an approximate stress-intensity-factor
$$\tilde a:=\frac 1\pi\int_{\Omega_0} (\Lambda\Gamma) |\tilde \omega_0|^3+\Delta(\Lambda\Gamma)\tilde \omega_0\,dx.$$
The regular part $v=u-aw\in H^2(\Omega_0)\cap H_0^1(\Omega_0)$ of a solution $u$ to \eqref{cbvp} satisfies
\beq
\label{regparteq}
-\Delta v=|aw+v|^3+a\Delta w
\eeq
and thus an approximate regular part $\tilde v$ (in the finite element space) can be computed using a Newton iteration
with initial guess $v_0=\tilde\omega_0-\tilde aI(w)$, where $I$ denotes the interpolation operator for the finite
element space. Our new approximate solution to \eqref{cbvp} is then given by
\beq
\label{newapp}
\omega_0 =\tilde aw+\tilde v,
\eeq
which yields an approximate solution $\omega$ to \eqref{bvp} by \eqref{omdef}.

\section{Computation of a residual bound}
Let $\tilde\rho\in H(\dive,\Omega)=\left\{ v\in (L^2(\Omega))^2:\ \dive v\in L^2(\Omega)\right \}$ be an
approximation of $\nabla\omega$, such that $\tilde \rho$ is also an approximate solution of $\dive\rho=-|\omega|^3$. We comment on the computation of $\tilde\rho$ in subsection \ref{defectcompnum}.
Then
we can estimate:
\begin{align*}
\|-\Delta\omega-|\omega|^3\|_{H^{-1}}& \leq \|-\dive(\nabla\omega-\tilde \rho)\|_{H^{-1}}+\|-\dive\tilde
\rho-|\omega|^3\|_{H^{-1}}\\
& \leq \|\nabla\omega-\tilde \rho\|_{L^2}+C_2\|-\dive\tilde \rho-|\omega|^3\|_{L^2},
\end{align*}
with $C_2$ denoting an embedding constant for the embedding $H_0^1(\Omega)\hookrightarrow L^2(\Omega)$ and hence also for 
$L^2(\Omega)\hookrightarrow H^{-1}(\Omega)$. Now we are left to compute upper bounds just for
integrals, which due to the splitting of the approximate solution into singular and regular part is however still technically a bit
challenging. We will comment on this in section \ref{defectcompnum}.

\begin{remark}
If $-\Delta \omega-|\omega|^3\in L^2(\Omega)$ (which requires higher smoothness of the approximate solution), one can
also use the embedding $L^2(\Omega)\hookrightarrow H^{-1}(\Omega)$ directly to compute an upper bound for the residuum:
$$\|-\Delta\omega-|\omega|^3\|_{H^{-1}}\leq C_2\|-\Delta\omega-|\omega|^3\|_{L^2}.$$
Again, it remains to compute bounds for an integral.
\end{remark}

\section{Computation of K}
We use the isometric isomorphism $\Phi:H_0^1(\Omega)\to H^{-1}(\Omega)$, defined in \eqref{Phidef}, to obtain
$$\|L_\omega[u]\|_{H^{-1}}=\|(\Phi^{-1}L_\omega)[u]\|_{H_0^1}\qquad (u\in H_0^1(\Omega))$$
and therefore
\beq
\label{KcondH10}
\|v\|_{H_0^1}\leq K\|L_\omega[v]\|_{H^{-1}}\quad (v\in H_0^1(\Omega))\iff \|v\|_{H_0^1}\leq
K\|(\Phi^{-1}L_\omega)[v]\|_{H_0^1}\quad (v\in H_0^1(\Omega)).
\eeq
Moreover, \eqref{PhiLsymm} already showed that $\Phi^{-1}L_\omega:H_0^1(\Omega)\to H_0^1(\Omega)$ is symmetric
and due to its definition on the whole space $H_0^1(\Omega)$ therefore self-adjoint. Thus \eqref{KcondH10} holds for any
\beq
\label{Kcondmin}
K\geq\frac 1{\min\{|\nu|:\ \nu \text{ is in the spectrum of }\Phi^{-1}L_\omega\}},
\eeq
provided the minimum is positive. We are therefore left to compute bounds for the essential spectrum of the operator
$\Phi^{-1}L_\omega$ as well as bounds for the eigenvalues of finite multiplicity which are closest to $0$. We first draw our attention to the
essential spectrum.

Consider the operator $L_0:H_0^1(\Omega)\to H^{-1}(\Omega),\ v\mapsto -\Delta v+\left
(\frac{\pi^2}{\pi^2+1}\chi_{\Omega_1}\right )v$, with $\chi_{\Omega_1}$ denoting the characteristic function on
$\Omega_1:=(-1,0)\times(0,1)$. Since both $\omega$ and $\chi_{\Omega_1}$ have compact support and are bounded,
$\Phi^{-1}L_\omega-\Phi^{-1}L_0:H_0^1(\Omega)\to H_0^1(\Omega)$ is compact, and hence
$\sigma_{\text{ess}}(\Phi^{-1}L_\omega)=\sigma_{\text{ess}}(\Phi^{-1}L_0)$ due to a well-known perturbation result
\cite{Kato}. To bound $\sigma_{\text{ess}}(\Phi^{-1}L_0)$ we consider Rayleigh quotients: $\Omega\backslash\Omega_1$
is the union of two disjoint semi-infinite strips, on each of which the Rayleigh quotient $\tfrac{\|\nabla
u\|_{L^2}^2}{\|u\|_{L^2}^2}$ is bounded from below by $\pi^2$. Hence, for each $u\in H_0^1(\Omega)$,
\beq
\label{specest1}
\int_{\Omega\backslash\Omega_1}|\nabla u|^2\,dx\geq\frac{\pi^2}{\pi^2+1}\int_{\Omega\backslash\Omega_1} \left
[|\nabla u|^2+u^2\right ]\,dx.
\eeq
Furthermore, trivially
\beq
\label{specest2}
\int_{\Omega_1}\left [|\nabla u|^2+\tfrac{\pi^2}{\pi^2+1}u^2\right ]\,dx\geq \frac{\pi^2}{\pi^2+1}\int_{\Omega_1}
\left [|\nabla u|^2+u^2\right ]\,dx
\eeq
holds. Adding \eqref{specest1} and \eqref{specest2} gives, for each $u\in H_0^1(\Omega)$,
$$\int_\Omega\left [|\nabla u|^2+\left (\tfrac{\pi^2}{\pi^2+1}\chi_{\Omega_1}\right )u^2\right ]\,dx 
\geq \frac{\pi^2}{\pi^2+1}\langle u,u\rangle_{H_0^1},$$
and the left-hand side equals $\langle \Phi^{-1}L_0u,u\rangle_{H_0^1}$. So the Rayleigh quotient, and hence the
spectrum, and in particular the essential spectrum of $\Phi^{-1}L_0$ is bounded from below by
$\tfrac{\pi^2}{\pi^2+1}$. Hence also $\sigma_{\text{ess}}(\Phi^{-1}L_\omega)\subset\left
[\frac{\pi^2}{\pi^2+1},\infty\right )$.

For analyzing eigenvalues of $\Phi^{-1}L_\omega$ we note that, for $(\nu,u)\in\R\times H_0^1(\Omega)$, $\nu \neq 1$,
\begin{align}
(\Phi^{-1}L_\omega)[u]=\nu u& \iff L_\omega[u]=\nu\Phi[u]\nonumber\\
& \iff -\Delta u-3|\omega|\omega u=\nu(-\Delta u+ u)\nonumber\\
& \iff (1-\nu)(-\Delta u+u)=(1+3|\omega|\omega)u\nonumber\\
& \iff (-\Delta u+ u)=\underbrace{\frac 1{1-\nu}}_{=:\kappa}(1+3|\omega|\omega)u,
\nonumber \\
& \iff \label{wevp}
\underbrace{\int_\Omega\left [\nabla u\cdot\nabla\varphi+ u\varphi\right ]\,dx}_{=\langle u,\varphi\rangle_{H_0^1}}=
\kappa\underbrace{\int_\Omega(1+3|\omega|\omega)u\varphi\,dx}_{=:N(u,\varphi)}\quad\text{for all }\varphi\in
H_0^1(\Omega),
\end{align}
which gives a new eigenvalue problem avoiding $\Phi^{-1}$, with spectral parameter $\kappa$. $N$ is a symmetric
bilinear form on $H_0^1(\Omega)$ and due to the positivity of $\omega$, which can be proved by interval evaluations, also positive definite. Therefore, $1-\nu>0$ for all possible eigenvalues $\nu$ and we are now left to compute
upper and lower bounds for eigenvalues $\kappa$ of \eqref{wevp} neighbouring $1$. 
Defining the essential spectrum of \eqref{wevp} in the usual way to be the one of its associated self-adjoint operator
$R=\left (I_{H_0^1}-\Phi^{-1}L_\omega\right)^{-1}$, we see that it is bounded from below by 
$(1-\min\sigma_{\text{ess}}(\Phi^{-1}L_\omega))^{-1}\geq\pi^2+1$.

The following theorem is well known and provides an easy and efficient way for computing upper bounds to eigenvalues below the essential spectrum, and hence (here) in particular to eigenvalues below $\pi^2+1$:
\begin{theorem}[{\bf Rayleigh-Ritz}]
\label{rayleighritz}
Let $v_1,\ldots, v_n\in H_0^1(\Omega)$ be linearly independent and define the matrices
$$A_0:=\left (\langle v_i,v_j\rangle_{H_0^1}\right)_{i,j=1,\ldots,n},\qquad 
A_1:=\left (N(v_i,v_j)\right)_{i,j=1,\ldots,n}.
$$
Denote by $\Lambda_1\leq\ldots\leq\Lambda_n$ the eigenvalues of $A_0x=\Lambda A_1x$ and suppose 
that $\Lambda_n<\pi^2+1$. Then there are at least $n$ eigenvalues of \eqref{wevp} below $\pi^2+1$, and the $n$ smallest of these, ordered by magnitude, satisfy
$$\kappa_i\leq\Lambda_i,\quad i=1,\ldots, n.$$
\end{theorem}
Note that good upper bounds will be obtained by Theorem \ref{rayleighritz} if $v_1,\ldots,v_n\in H_0^1(\Omega)$ are
chosen as approximate eigenfunctions associated with the $n$ smallest eigenvalues of \eqref{wevp}. The remaining task for applying Theorem \ref{rayleighritz} is the
enclosure of matrix eigenvalues, which can be achieved using interval arithmetic and
\cite[Lemma 4]{HoangPlWie} or using interval packages like INTLAB \cite{Ru99a}.

For the computation of lower eigenvalue bounds, which is more problematic than obtaining upper bounds, we use the
following method of Lehmann and Goerisch (see \cite{BehnkeGoe}). 
\begin{theorem}
\label{lehmanngoerisch}
Let $v_1,\ldots,v_n\in H_0^1(\Omega)$ and $A_0,A_1$ as before. Let $X$ be some vector space, $b$ some symmetric,
positive definite bilinear form on $X$, and $T:H_0^1(\Omega)\to X$ some linear operator satisfying
$b(T\psi,T\varphi)=\langle\psi,\varphi\rangle_{H_0^1}$ for all $\psi,\varphi\in H_0^1(\Omega)$.

Let $w_1,\ldots, w_n\in X$ satisfy
\beq
\label{bcond}
b(T\varphi,w_i)=N(\varphi,v_i)\quad\text{for all }\varphi\in H_0^1(\Omega),\ i=1,\ldots,n
\eeq
and define $A_2:=\left (b(w_i,w_j)\right)_{i,j=1\ldots,n}$. Moreover, let $\rho\in\R$ such that
\beq
\label{rhocond1}
\Lambda_n<\rho\leq\pi^2+1
\eeq
and in addition 
\beq
\label{rhocond2}
\rho\leq\kappa_{n+1},
\eeq
if an $(n+1)$-st eigenvalue $\kappa_{n+1}<\pi^2+1$ exists.\\
Then, with $\mu_1\leq\ldots\leq\mu_n<0$ denoting the eigenvalues of 
\beq
\label{lehgoeevp}
(A_0-\rho A_1)x=\mu(A_0-2\rho A_1+\rho^2 A_2)x,
\eeq we have
$$\kappa_m\geq\rho-\frac\rho{1-\mu_{n+1-m}},\quad m=1,\ldots, n.$$
\end{theorem}

\begin{remark}
\label{lehgoeremark}
\begin{itemize}
\item[(i)] \eqref{rhocond1} and Theorem \ref{rayleighritz} imply in particular that at least $n$ eigenvalues
$\kappa_1\leq\ldots\leq\kappa_n<\pi^2+1$ exist.
\item[(ii)] Again by \eqref{rhocond1} and Theorem \ref{rayleighritz}, the matrix $A_0-\rho A_1$ is negative definite. Moreover, \eqref{bcond} and \eqref{rhocond1} show after some calculations that $A_0-2\rho A_1+\rho^2A_2$ is positive definite, hence the matrix eigenvalue problem 
\eqref{lehgoeevp} has indeed only negative eigenvalues.
\item[(iii)] We will see later that usually \eqref{bcond} does not determine $w_1,\ldots, w_n$ uniquely. A closer look
at the proof of Theorem \ref{lehmanngoerisch} makes clear that good bounds will be obtained if $w_i\approx \frac
1{\Lambda_i} Tv_i$, when $(\Lambda_i, v_i)$ is an approximate eigenpair to problem \eqref{wevp}.
\item[(iv)] Condition \eqref{rhocond2} requires an a-priori lower bound for the $(n+1)$-st eigenvalue (if it
exists) in order to compute lower bounds for the $n$ smallest eigenvalues. However, a rough lower bound $\rho$ will be
sufficient for this purpose and can be obtained using a homotopy method (see subsection \ref{sechomotopy}).
\end{itemize}
\end{remark}
We will now explain how to choose $X,b,T$ and $w_1,\ldots,w_n\in X$ satisfying the assumptions of Theorem \ref{lehmanngoerisch}: Let
$$X=\left (L^2(\Omega)\right)^2\times L^2(\Omega),\ b\left (\bpm w_1\\w_2\epm,\bpm\tilde w_1\\\tilde w_2\epm\right):=
\langle  w_1,\tilde w_1\rangle_{L^2}+\langle w_2,\tilde w_2\rangle_{L^2},\ Tu:=\bpm\nabla u\\u\epm.$$
Obviously, $b(T\psi,T\varphi)=\langle\psi,\varphi\rangle_{H_0^1}$ for all $\psi,\varphi\in H_0^1(\Omega)$. We consider
condition \eqref{bcond}:
\begin{align}
& b(T\varphi,w_i)=N(\varphi,v_i)\quad\text{for all }\varphi\in H_0^1(\Omega)\nonumber\\
\iff \ & \langle\nabla\varphi,w_{i,1}\rangle_{L^2}+\langle \varphi,w_{i,2}\rangle_{L^2}=\langle\varphi,
(1+3|\omega|\omega)v_i\rangle_{L^2}\quad\text{for all }\varphi\in H_0^1(\Omega)\nonumber\\
 \iff \ &w_{i,1}\in H(\dive,\Omega),\quad -\dive(w_{i,1})+w_{i,2}=(1+3|\omega|\omega)v_i\nonumber\\
\iff \ & w_{i,1}\in H(\dive,\Omega),\quad w_{i,2}=\dive(w_{i,1})+(1+3|\omega|\omega)v_i. 
\label{wi2cond}
\end{align}
This shows that $w_{i,1}\in H(\dive,\Omega)$ can be chosen arbitrarily whereas $w_{i,2}$ has to be chosen according to
\eqref{wi2cond}. Recalling Remark \ref{lehgoeremark} (iii), one should have $w_i\approx \frac 1{\Lambda_i}Tv_i$ with
an approximate eigenpair $(\Lambda_i,v_i)$ to obtain good bounds. Since $w_{i,2}$ is already fixed, it remains to require
$$w_{i,1}\approx\frac 1{\Lambda_i}\nabla v_i.$$
A suitable choice of $w_{i,1}$ is therefore given by an approximate minimizer in $H(\dive,\Omega)$ of 
$$\bigl \|\tfrac 1{\Lambda_i}\nabla v_i-w\bigr \|_{L^2}^2+
\bigl \|-\dive w+\bigl (\tfrac 1{\Lambda_i}-(1+3|\omega|\omega)\bigr )v_i\bigr \|_{L^2}^2.$$

\subsection{Homotopy method}
\label{sechomotopy}
Our aim is now to find some $\rho\in\R$ such that $\Lambda_n<\rho\leq\kappa_{n+1}<\pi^2+1$ if
$\kappa_{n+1}$ exists, or $\Lambda_n<\rho\leq\pi^2+1$ otherwise. The crucial idea is to find a base
problem, for which we have knowledge about the eigenvalues, and connect it with the original problem via a family of
eigenvalue problems such that, indexwise, the eigenvalues increase along the homotopy. 
In our case it is necessary to combine two separate homotopies: One is needed to find lower bounds for eigenvalues of
the eigenvalue problem $-\Delta u+ u=\kappa(1+\overline c)u$ in $\Omega$ ($u\in H_0^1(\Omega)$), where
$\overline c$ is a suitable piecewise constant function on $\Omega$. A second homotopy then connects this
eigenvalue problem to the original one. For the first homotopy we use a domain decomposition method, which goes
back to an idea of E.B. Davies and is explained in detail in \cite{BehMerPlWie}. 

To construct a suitable base problem, choose $0=\xi_0<\xi_1<\ldots<\xi_k$ \,and a function $\overline c\geq
3|\omega|\omega$ which is constant on each of the rectangles
$$
\begin{array}{crl}
(-1,0)\times(0,1),& (\xi_i,\xi_{i+1})\times(0,1)\quad (i=0,\ldots, k-1),& (\xi_k,\infty)\times(0,1)\\
& (-1,0)\times(-\xi_{i+1},-\xi_i)\quad (i=0,\ldots,k-1),& (-1,0)\times(-\infty,-\xi_k).
\end{array}
$$
Note that since $\omega$ has compact support, $\overline c=0$ on $(\xi_k,\infty)\times(0,1)$ and on
$(-1,0)\times(-\infty,-\xi_k)$ can be chosen if $\xi_k$ is large enough (which we will assume in the following). Define
now $\Omega_1:=(-1,0)\times(0,1),\ \Omega_2:=(0,\infty)\times(0,1),\ \Omega_3:=(-1,0)\times(-\infty,0)$ and consider
for $j=1,2,3$ the eigenvalue problems 
\beq
\label{domdecevp}
\left \{\begin{array}{rcll}
-\Delta u+ u & = & \kappa(1+\overline c)u& \text{in } \Omega_j\\
\pd u\nu & = & 0 & \text{on the interfaces }\Gamma_1:=\{0\}\times(0,1),\ \Gamma_2:=(-1,0)\times\{0\}\\
u & = & 0 & \text{on }\partial\Omega\cap\partial\Omega_j.
\end{array}\right .
\eeq
Note that a lower bound for the essential spectrum of the eigenvalue problems on $\Omega_2$ and $\Omega_3$ is given by
$\pi^2+1$ due to the compact support of $\overline c$. Using separation of variables on each
of the above rectangles we can compute fundamental systems for the resulting ODE problems. Problem \eqref{domdecevp} (for $j = 1,2,3$)  thus leads to transcendental equations in $\kappa$, whose solutions are the eigenvalues of
\eqref{domdecevp}. By interval bisection and an Interval Newton method we can compute enclosures of these roots and therefore enclosures for all eigenvalues below $\pi^2+1-\eps$, with some appropriately chosen $\eps>0$, of the three eigenvalue problems.
Let $\kappa_1^{(0)}\leq\kappa_2^{(0)}\leq\ldots\leq\kappa_L^{(0)}$ denote the union of all these eigenvalues (of all three problems) ordered by magnitude and counted by multiplicity. The following Lemma allows us to compare these eigenvalues with the eigenvalues below $\pi^2 + 1$ of the problem
\beq
\label{evpconstcoeff}
\left \{\begin{array}{rcll}
-\Delta u+u & = & \kappa(1+\overline c)u& \text{in }\Omega\\
u & = & 0 & \text{on }\partial\Omega,
\end{array}\right .
\eeq
which we denote by $\kappa_i^{(\infty)}$, ordered by magnitude and counted by multiplicity.
\begin{lemma}
\label{basebaseev}
For all $i=1,\ldots, L$ we have $\kappa_i^{(0)}\leq\kappa_i^{(\infty)}$, provided that an $i$-th eigenvalue 
$\kappa_i^{(\infty)}<\pi^2+1$ of \eqref{evpconstcoeff} exists.
\end{lemma}
\begin{proof}
Let $V:=\left \{u\in L^2(\Omega):\ u|_{\Omega_j}\in H^1(\Omega_j),\ u|_{\partial\Omega\cap\partial\Omega_j}=0 \text{
for } j=1,2,3\right \}$. Since $V\supset H_0^1(\Omega)$ we have due to Poincaré's min-max principle:
\begin{align*}
\kappa_i^{(0)}&=\inf_{\substack{{U\subset V\text{ subspace}}\\{\dim U=i}}}\max_{u\in U\backslash\{0\}} 
\frac{\langle \nabla u,\nabla u\rangle_{L^2}+\langle u,u\rangle_{L^2}}{\langle
(1+\overline c)u,u\rangle_{L^2}}\\ &\leq
\inf_{\substack{{U\subset H_0^1(\Omega)\text{ subspace}}\\{\dim U=i}}}\max_{u\in U\backslash\{0\}} 
\frac{\langle \nabla u,\nabla u\rangle_{L^2}+\langle u,u\rangle_{L^2}}{\langle
(1+\overline c)u,u\rangle_{L^2}}=\kappa_i^{(\infty)}.
\end{align*}
\end{proof}
In principle, we can construct a homotopy connecting problems \eqref{domdecevp} and \eqref{evpconstcoeff}, but this is unnecessary, since a pure comparison of these two problems already leads to the desired rough lower bound for some higher 
eigenvalue of \eqref{evpconstcoeff}. More precisely, numerical Rayleigh-Ritz computations (Theorem \ref{rayleighritz}) for problem \eqref{evpconstcoeff}, with some suitably chosen $n \le L-1$, turn out to give bounds $\bar{\kappa}_1^{(\infty)} \le \dots \le \bar{\kappa}_{n+1}^{(\infty)} < \pi^2 + 1$, whence by Theorem \ref{rayleighritz} at least $n+1$ eigenvalues $\kappa_1^{(\infty)} \le \dots \le \kappa_{n+1}^{(\infty)}$ of problem \eqref{evpconstcoeff} below $\pi^2 + 1$ exist, and $\kappa_i^{(\infty)} \le \bar{\kappa}_i^{(\infty)}$. Moreover, the computations give $\bar{\kappa}_n^{(\infty)} < \kappa_{n+1}^{(0)}$, whence we can find some $\rho\in\R$ such that
\beq
\label{evpconstcoeff_a}
\bar{\kappa}_n^{(\infty)} < \rho\leq\kappa_{n+1}^{(0)}\stackrel{\text{Lemma }\ref{basebaseev}}\leq 
\kappa_{n+1}^{(\infty)}  \,< \, \pi^2 + 1.
\eeq
As a second step we have to connect problem \eqref{evpconstcoeff}, with piecewise constant coefficient function on the right-hand-side, with our original eigenvalue problem \eqref{wevp}. For this purpose we define $c_s:=(1-s)\overline
c+s\left (3|\omega|\omega\right )$ for $0\leq s\leq 1$ and consider the $s$-dependent eigenvalue problem
\beq
\label{sevp}
\int_\Omega \left[\nabla u\cdot\nabla\varphi+u\varphi \right ]\,dx=\tilde \kappa^{(s)}\int_\Omega(1+c_s)u\varphi\,dx
\quad\text{for all }\varphi\in H_0^1(\Omega).
\eeq
Obviously, for $s=0$ this eigenvalue problem equals \eqref{evpconstcoeff} (i.e. $\tilde{\kappa}_i^{(0)} = \kappa_i^{(\infty)}$) and for $s=1$ we have problem \eqref{wevp}. 
Moreover, by Poincaré's min-max principle, the eigenvalues of \eqref{sevp} increase along the homotopy since $\overline c\geq 3|\omega|\omega$, i.e.
for $0\leq s\leq t\leq 1$ we have $\tilde \kappa_i^{(s)}\leq\tilde\kappa_i^{(t)}$ as long as
$\tilde\kappa_i^{(t)}<\pi^2+1$. Step by step, this provides numbers $\rho$ (see \eqref{rhocond1}, \eqref{rhocond2}) for the application of Theorem \ref{lehmanngoerisch} to problem \eqref{sevp}, for an increasing (finite) sequence of $s$-values. A detailed description of this homotopy method can be found e.g. in
\cite{Breuerbeam}.

\section{Numerical Results}
\subsection{Practical computation of the residuum}
\label{defectcompnum}
We first want to comment on some technical difficulties arising in the defect computation. Recall that we can estimate
the residuum by:
\begin{align}
\label{resdef}
\|-\Delta\omega-|\omega|^3\|_{H^{-1}} \leq \|\nabla\omega-\tilde \rho\|_{L^2}+C_2\|-\dive\tilde
\rho-|\omega|^3\|_{L^2},
\end{align}
where $\tilde\rho\in H(\dive,\Omega)$ approximately minimizes the right-hand side of \eqref{resdef} and therefore is an approximation of $\nabla\omega$. The approximate solution $\omega$ is of the form $\omega=\tilde aw+v$, 
with a corner singular function $w=\lambda\gamma\in H_0^1(\Omega)$ (where $\gamma$ is given by \eqref{gammadef} and $\lambda(x,y)=(1-x^2)^2(1-y^2)^2\chi_{(-1,1)^2}$) and a finite element approximation $v\in
V_N$ of the regular part (compare \eqref{newapp}; now we write $v$ instead of $\tilde{v}$). Here $V_N$ denotes an $H_0^1$-conforming finite element space of dimension $N$. Let now $\tilde\rho_v\in
\left (V_N \right )^2$ be an approximation of $\nabla v$ (as well as an approximate solution of
$\dive\rho=-|\omega|^3-\tilde a\Delta w$; compare \eqref{regparteq}). $\tilde\rho_v$ is computed by approximate minimization of $\|\rho-\nabla v\|_{L^2}^2+C_2^2\|\dive\rho+\tilde a\Delta w+|\omega|^3\|_{L^2}^2$ in $(V_N)^2$. Define 
$$\tilde\rho:=\tilde a\nabla w+\tilde\rho_v.$$
Substituting the expressions for $\omega$ and $\tilde\rho$ in \eqref{resdef} yields
\beq
\|-\Delta\omega-|\omega|^3\|_{H^{-1}} \leq \|\nabla v-\tilde\rho_v\|_{L^2}+C_2\|-\tilde a\Delta
w-\dive\tilde\rho_v-|\tilde
aw+v|^3\|_{L^2}.
\label{resreg}
\eeq
The first summand on the right-hand side of \eqref{resreg} can be computed exactly using a quadrature rule of
sufficiently
high degree together with interval arithmetic, as $\tilde \rho$ is an element of $(V_N)^2$ and also $\nabla v$ is piecewise polynomial in each component. Due
to the mixture of cartesian and polar coordinates in the second summand, a verified computation of a tight upper bound for this term is technically non-trivial. We interpolate the singular function $\gamma(r,\varphi)$, as well as $x\pdt{\gamma}{x}(r,\varphi)=-\tfrac 23r^{2/3}\cos\varphi\sin\tfrac \varphi 3$ and $y\pdt\gamma y(r,\varphi)=\tfrac 23r^{2/3}\sin\varphi\cos\tfrac\varphi 3$ in the finite element space $V_N$, and replace the corresponding terms in the second summand in \eqref{resreg} by these interpolations. Now the integrand is piecewise polynomial, and hence the integral can be computed exactly. The remaining task is to bound the interpolation errors, which after some elementary estimations amounts to the computation of bounds to
\beq
\label{gammaint}
\int_K(\gamma-I(\gamma))^2\,d(x,y)
\eeq
for each element $K$ (with $I(\gamma)$ denoting the interpolation), and to analogous terms for $x\pdt \gamma x$ and $y\pdt\gamma y$. For this purpose we cover each element $K$ by a finite union of circular segments (i.e. rectangles in polar coordinates), replace $K$ in \eqref{gammaint} by this union (giving an upper bound to the integral), and transform the integral to polar coordinate integration (over a union of rectangles), which can be carried out in closed form, using Maple for calculating primitive functions.

\subsection{Computational Results}

We will now report on some of the numerical results that finally prove the existence of a solution of problem
\eqref{bvp} together with an error bound. All computations have been carried out on the parallel cluster OTTO of the Institute for Applied and Numerical Mathematics at Karlsruhe Institute of Technology. We used the finite element software M$++$ \cite{Ref1}, which is written in C$++$. For interval arithmetic we used the libraries C-XSC \cite{Ref2} as well as MPFR and MPFI \cite{Ref3}. Our source code is available on {\it http://www.math.kit.edu/user/mi1/Plum/PaperPPR/} or upon request to the third author.

As already mentioned in section \ref{approxsol}, we used the computational domain $\Omega_0=\Omega\cap(-3,3)^2$, and taking symmetry into account we restricted ourselves to the half domain
$\tilde\Omega_0=\conv\{(0,0),(3,0),$\linebreak
$(3,1),(-1,1)\}$, imposing Neumann boundary conditions on
$\partial\tilde\Omega_0\backslash\partial\Omega_0=:\Gamma_N$. This diminishes the constant $K$ in \eqref{Kcond}, since only eigenfunctions which are symmetric w.r.t. $\Gamma_N$ have to be considered in the eigenvalue problem \eqref{wevp}, and hence less eigenvalues contribute to the minimum in \eqref{Kcondmin}. Furthermore, symmetry reduces the computational effort. Using the space $\{ u \in H^1 (\tilde{\Omega}_0): u = 0$ on $\partial \Omega_0 \cap \partial \tilde{\Omega}_0\}$ in all computations has lead to an approximate solution that is symmetric w.r.t. $\Gamma_N$ and eventually - after a successful application of our theoretical results - to a symmetric solution $u$ of \eqref{bvp}.

In our computations we used Serendipity finite elements, whose nodes are given by corners and midpoints of the elements.
Figure \ref{fig:appsol} shows the computed approximate solution $\omega_0\in H_0^1(\Omega_0)$ (see \eqref{omdef}) on the full domain $\Omega_0$.

\begin{figure}[h]
 \includegraphics[width=0.5\textwidth]{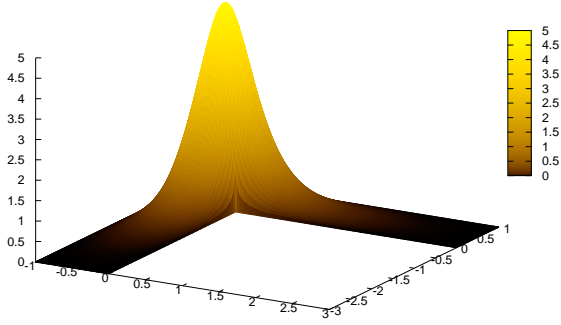}
\caption{Approximate solution $\omega_0\in H_0^1(\Omega_0)$}
\label{fig:appsol}
\end{figure}

To compute bounds for eigenvalues of \eqref{wevp} we used, as described before, a domain decomposition method to obtain lower bounds for the base
eigenvalues and a homotopy method to connect the base problem \eqref{evpconstcoeff} with our original eigenvalue problem \eqref{wevp}. We note that for the eigenvalue computation we did not use the approximation $\omega=\tilde a w+v$, but an interpolation
$\omega_{FE}:=I_{V_{\tilde N}}(\omega)$ in a finite element space $V_{\tilde N}$ which is coarser than
$V_N$. This avoids complicated integration during the homotopies and saves computation time. Eventually we obtain a
bound for the inverse of the linearization at $\omega_{FE}$, which can then be used to compute the corresponding bound
for $L_\omega$ by Lemma \ref{l1} (b), with $\tilde{u} = \omega_{FE}$ and $u = \omega$.

Recall that the
base problem is given by
\beq
\label{34}
\left \{\begin{array}{rcll}
-\Delta u+u & = & \kappa(1+\overline c)u& \text{in }\Omega\\
u & = & 0 & \text{on }\partial\Omega,
\end{array}\right .
\eeq
where now $\overline c\in L^\infty(\Omega)$ is chosen such that $\overline c\geq 3|\omega_{FE}|\omega_{FE}$ in $\Omega$
and is constant on the rectangles $(-1,0)\times(0,1)$, $(0,1)\times(0,1)$, $(1,3)\times (0,1)$ and
$(3,\infty)\times(0,1)$ as well as on $(-1,0)\times (-1,0),\ (-1,0)\times(-3,-1)$ and $(-1,0)\times(-\infty,-3)$. This
choice defines also the aforementioned comparison problem for the domain decomposition.

In the following table and figure we display some results of our eigenvalue computations. The first column of the table shows lower bounds $\underline{\kappa}_i^{(0)}$ for the eigenvalues $\kappa_i^{(0)}$ of the comparison problem, which by Lemma \ref{basebaseev}  constitute lower bounds (indexwise) for the eigenvalues $\kappa_i^{(\infty)} = \tilde{\kappa}_i^{(0)}$ of the
base eigenvalue problem \eqref{34}. The second column of the table shows upper bounds $\bar{\kappa}_i^{(\infty)}$ for the eigenvalues $\kappa_i^{(\infty)}$, computed by Theorem \ref{rayleighritz}. In particular $\bar{\kappa}_{11}^{(\infty)} < \underline{\kappa}_{12}^{(0)} \; (\le \kappa_{12}^{(\infty)})$, whence \eqref{evpconstcoeff_a} holds for $n=11$ and $\rho = \underline{\kappa}_{12}^{(0)}$, which enables the start of the homotopy for problem \eqref{sevp}. The figure shows the course of the homotopy, where we started with the a-priori lower bound $\underline{\kappa}_{12}^{(0)}$ for the $12$th eigenvalue $\tilde{\kappa}_{12}^{(0)} = \kappa_{12}^{(\infty)}$ of problem \eqref{sevp} with $s=0$ (and $\omega$ replaced by $\omega_{FE}$). For additional illustration, the figure also contains approximations to the eigenvalues $\tilde{\kappa}_1^{(0)}, \dots, \tilde{\kappa}_{12}^{(0)}$. At the end of the homotopy we obtained a lower bound for the third eigenvalue of problem \eqref{wevp} (with $\omega$ replaced by $\omega_{FE}$). 
Using the Lehmann-Goerisch Theorem once again we computed the desired lower bound for the second eigenvalue of \eqref{wevp} (with
$\omega_{FE}$ instead of $\omega$), which is the smallest eigenvalue above 1. Finally an application of the Rayleigh-Ritz method yields an upper bound for the first eigenvalue of \eqref{wevp} (with $\omega$ replaced by $\omega_{FE}$) and Lemma \ref{l1} (b) then gives a bound for the inverse of the linearization at $\omega$.

\begin{minipage}{0.3\textwidth}
\begin{tabular}{|c|c|c|}
\hline
$i$ & $\underline \kappa_i^{(0)}$ & $\overline \kappa_i^{(\infty)}$ \\
\hline\hline
1 & 0.08291 & 0.18117 \\
\hline
2 & 0.35867 & 0.48073 \\
\hline
3 & 0.52231 & 0.69818 \\
\hline
4 & 0.63443 & 0.84524 \\
\hline
5 & 0.91020 & 1.02708 \\
\hline
6 & 1.08005 & 1.43595 \\
\hline
7 & 1.18596 & 1.52509 \\
\hline
8 & 1.73749 & 1.93536 \\
\hline
9 & 1.73749 & 1.97598 \\
\hline
10 & 1.79188 & 2.07092 \\
\hline
11 & 2.01325 & 2.27739 \\
\hline
12 & 2.35653 & 2.75040  \\
\hline
\end{tabular} 
\vspace{0.3cm}

Table 1: Eigenvalues of\\ the comparison problem\\ and the base problem
\end{minipage}
\begin{minipage}{0.68\textwidth}
\psfrag{k1}{\!\scriptsize$\tilde\kappa_1^{(0)}$}
\psfrag{k2}{\!\scriptsize$\tilde\kappa_2^{(0)}$}
\psfrag{k3}{\!\scriptsize$\tilde\kappa_3^{(0)}$}
\psfrag{k4}{\!\scriptsize$\tilde\kappa_4^{(0)}$}
\psfrag{k5}{\!\scriptsize$\tilde\kappa_5^{(0)}$}
\psfrag{k6}{\!\scriptsize$\tilde\kappa_6^{(0)}$}
\psfrag{k7}{\!\scriptsize$\tilde\kappa_7^{(0)}$}
\psfrag{k8}{\!\scriptsize$\tilde\kappa_8^{(0)}$}
\psfrag{k9}{\!\scriptsize$\tilde\kappa_9^{(0)}$}
\psfrag{k10}{\!\scriptsize$\tilde\kappa_{10}^{(0)}$}
\psfrag{k11}{\!\scriptsize$\tilde\kappa_{11}^{(0)}$}
\psfrag{k12}{\!\scriptsize$\tilde\kappa_{12}^{(0)}$} 
\psfrag{k12low}{\!\!\!$\mathbf{\underline{\kappa}_{12}^{(0)}}$}
\psfrag{0}{}
\psfrag{1}{\scriptsize$1$}
\psfrag{s0}{\!\!\tiny$s=0$}
\psfrag{s1}{\!\!\!\tiny$s_1=0.0313$}
\psfrag{s2}{\!\!\!\tiny$s_2=0.0918$}
\psfrag{s3}{\!\!\!\tiny$s_3=0.3472$}
\psfrag{s4}{\!\!\!\tiny$s_4=0.643$}
\psfrag{s5}{\!\!\!\tiny$s_5=0.6876$}
\psfrag{s6}{\!\!\!\tiny$s_6=0.8438$}
\psfrag{s7}{\!\!\!\tiny$s=1$}
\psfrag{2.5}{}
\psfrag{k3l}{\scriptsize$\underline\kappa_3$}
\psfrag{k2l}{\scriptsize$\mathbf{\begin{array}{r} \mathbf{\underline\kappa_2=}\\ \mathbf{1.368} \end{array}}$}
\psfrag{k1u}{\scriptsize$\mathbf{\begin{array}{r} \mathbf{\overline\kappa_1=}\\ \mathbf{0.353 } \end{array}}$}
\includegraphics[width=\textwidth]{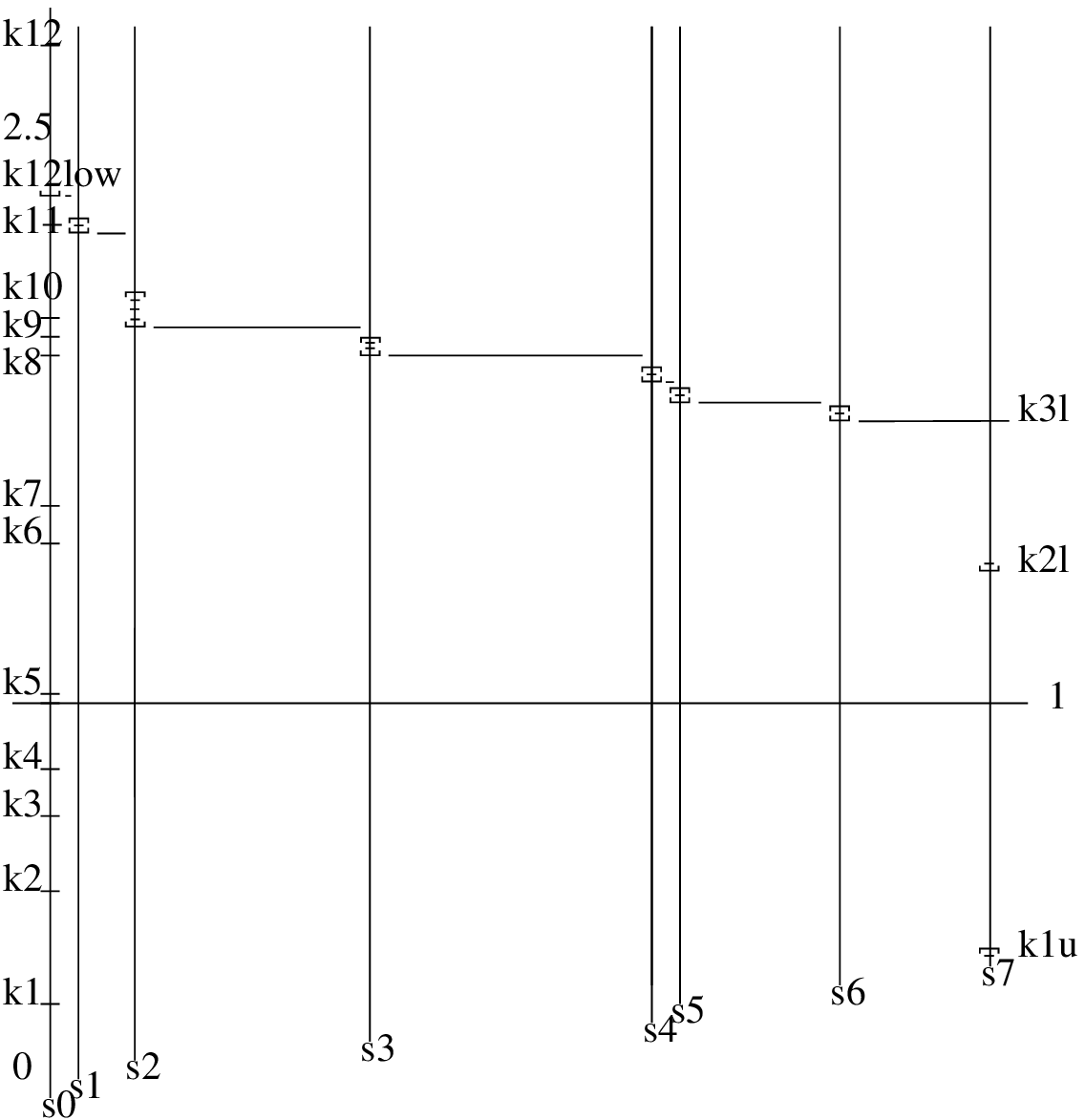}
\begin{center}
Figure 2: Course of the homotopy
\end{center}
\end{minipage}

An embedding constant $C_4$ for the embedding $H_0^1(\Omega)\hookrightarrow L^4(\Omega)$, which is needed for the application of Theorem \ref{maintheo}, can be computed using Lemma 2 in \cite{PlumDMV}. The computation requires a lower bound for the smallest Dirichlet eigenvalue of $-\Delta$ on $\Omega$, which can again be obtained using the Lehmann-Goerisch method.

Verified bounds for some of the relevant data are given by:
\begin{align*}
\|\omega\|_{L^4(\Omega)} &\leq 3.014333 \\ 
\|-\Delta\omega-|\omega|^3\|_{H^{-1}(\Omega)} & \leq 0.001699\\ 
K^{\text{sym}} & = 3.722884\\ 
C_4 & = 0.462000\\ 
\alpha & = 0.006471, \\ 
\end{align*}
where $K^{\text{sym}}$ denotes a constant satisfying
\beq
\label{Ksym}
\|v\|_{H_0^1(\Omega)}\leq K^{\text{sym}}\|L_\omega[v]\|_{H^{-1}(\Omega)}\qquad\text{for all }v\in H_0^1(\Omega)
\text{ which are symmetric w.r.t. }\Gamma_N.
\eeq
Thus, the existence of a symmetric solution $u\in H_0^1(\Omega)$ to problem \eqref{bvp} with
$\|u-\omega\|_{H_0^1}\leq 0.006471$ is proved.

\section{Nondegeneracy of the solution}

Besides decay properties, which we will analyze in the next section, we also prove nondegeneracy of the
solution $u\in H_0^1(\Omega)$ we have obtained, i.e. that $0$ is not in the spectrum of the linearization $L_u$ at $u$. For this purpose we note that due to the restriction to the half domain $\tilde\Omega_0$ with Neumann boundary conditions on $\Gamma_N$,
the eigenvalue bounds computed in the previous section are not sufficient to prove nondegeneracy of $u$ in
the whole space $H_0^1(\Omega)$. Indeed, we have to compute also eigenvalues of $\Phi^{-1}L_\omega:H_0^1(\Omega)\to
H_0^1(\Omega)$ corresponding to antisymmetric eigenfunctions. This can be done using the same methods as described in the
previous sections, this time imposing Dirichlet boundary conditions on $\Gamma_N$. Altogether, bounds for the 
smallest eigenvalues of $\Phi^{-1}L_\omega$ are given by
\beq
\label{ev12bounds}
\nu_1\leq -1.8359,\qquad \nu_2\geq 0.1116,
\eeq
where the first eigenvalue corresponds to a symmetric eigenfunction, while the second one corresponds to an
antisymmetric one. Thus, $0$ is not in the spectrum of $\Phi^{-1} L_{\omega}$. To obtain the corresponding result also for $\Phi^{-1} L_u$, we first apply Lemma \ref{l1} (b), with $p_1=p_2=p_3=p_4=4$, to obtain
\begin{corollary}
\label{cor1}
Let $u,\tilde u\in H_0^1(\Omega)$ and let $\dist\left (\sigma(\Phi^{-1}L_{\tilde u}),0\right)>0$. Furthermore assume
that, for some $\alpha>0$,
$$\|u-\tilde u\|_{H_0^1}\leq\alpha.$$
If
\beq
\label{kappanondeg}
\kappa:=3C_4^3 \left (C_4\alpha+2\|\tilde u\|_{L^4}\right)\alpha<\dist(\sigma(\Phi^{-1}L_{\tilde u}),0)
\eeq
then nondegeneracy of $u$ follows.
\end{corollary}
For the proof note that $\dist(\sigma(\Phi^{-1}L_{\tilde u}),0)>0$ implies 
$\|v\|_{H_0^1}\leq K\|L_{\tilde u}[v]\|_{H^{-1}}$ ($v\in H_0^1(\Omega)$) with $K=\tfrac 1{\dist(\sigma(\Phi^{-1}L_{\tilde u}),0)}$ (see also \eqref{KcondH10}, \eqref{Kcondmin}).
 
We apply Corollary \ref{cor1} with $\tilde u=\omega$, and $u$ being the solution of \eqref{bvp} given
by Theorem \ref{maintheo}, and thus satisfying 
\beq
\label{encl2}
\|u-\omega\|_{H_0^1}\leq\alpha.
\eeq
Using the data of the previous paragraph we obtain $\kappa=0.01473$ and therefore, by \eqref{ev12bounds} and \eqref{kappanondeg}, nondegeneracy of $u$. Moreover, using Poincaré's min-max principle, we can deduce from \eqref{ev12bounds} and \eqref{encl2} that the Morse index of $u$ is 1; we omit the details here.

\section{Decay at infinity}

We finally prove that the solution $u\in H_0^1(\Omega)$ decays exponentially at infinity.

\begin{theorem}
Let $\phi(z)=\sin(\pi z)$. There exists a constant $\beta>0$ such that for $x\to\infty$ we have
\begin{align*}
u(x,y)&=\beta e^{-\pi x}\phi(y)+o(e^{-\pi x}\phi(y))\\
\pdt ux(x,y)&=-\beta \pi e^{-\pi x}\phi(y)+o(e^{-\pi x}\phi(y))
\end{align*}
and analogously for $y\to-\infty$
\begin{align*}
u(x,y)&=\beta e^{\pi y}\phi(x+1)+o(e^{\pi y}\phi(x+1))\\
\pdt uy(x,y)&=\beta \pi e^{\pi y}\phi(x+1)+o(e^{\pi y}\phi(x+1)).
\end{align*}
\end{theorem}
\begin{proof}
For the proof it is sufficient to show that $u\in C(\,\overline\Omega\,\backslash [-2.5,2.5]^2)$ and $u(x,y)\to 0$ as
$x\to\infty$ uniformly in $y$ (and similarly $u(x,y)\to 0$ as $y \to -\infty$ uniformly in $x$). Then Prop. 4.2. in
\cite{BerNir} can be applied to the half strips $\Omega\cap\{x>2.5\}$ and $\Omega\cap\{y<-2.5\}$, and implies the theorem. 

So let $r=\tfrac 1{\sqrt{2}}+\eps$ (for some sufficiently small $\eps>0$) and let
$\Omega_c\subset\Omega$ be a $C^2$-domain such that
$\{(x,y)\in\Omega:\
2.5<x<3.5\}\subset \Omega_c\subset \{(x,y)\in\Omega:\ 2<x<4\}$ and $\dist\left (\Omega_c\cap B_r(3,\tfrac
12),(\partial\Omega_c)\cap \Omega\right )>0$. Choose moreover a cutoff-function $\zeta\in
C^\infty(\overline\Omega_c)$ such that
$\zeta=0$ on $(\partial\Omega_c)\cap\Omega$ and $\zeta=1$ in $\Omega_c\cap B_r(3,\tfrac 12)$.

Define now $\tilde u=u\zeta\in H_0^1(\Omega_c)$; then for every $\varphi\in H_0^1(\Omega_c)$ we have, since 
$\int_{\Omega_c}\nabla u\cdot\nabla (\zeta \varphi)\,dx=\int_{\Omega_c}u^3 \zeta \varphi\,dx$,
\begin{align*}
\int_{\Omega_c}\nabla\tilde u\cdot\nabla \varphi\,dx & =\int_{\Omega_c}\left (\zeta\nabla u\cdot\nabla\varphi+u\nabla\zeta\cdot\nabla\varphi\right )\,dx\\
& =\int_{\Omega_c}\nabla u\cdot \nabla (\zeta\varphi)\,dx-\int_{\Omega_c}\left (2\nabla u\cdot\nabla\zeta+u\Delta\zeta\right )\varphi\,dx
=\int_{\Omega_c}f\varphi\,dx,
\end{align*}
where $f:=u^3\zeta-2\nabla u\cdot\nabla\zeta-u\Delta \zeta\in L^2(\Omega_c)$.

Then Theorem 9.8 in \cite{Agmon} yields (with $\Omega_1:=\Omega_c\cap B_r(3,\tfrac12)$):
$$\tilde u\in H^2(\Omega_1) \quad\text{and}\quad 
\|\tilde u\|_{H^2(\Omega_1)}\leq\gamma\left (\|\tilde u\|_{L^2(\Omega_c)}+\|f\|_{L^2(\Omega_c)}\right )
$$
for some constant $\gamma$ not depending on $u$.

Since $\tilde u=u$ in $\Omega_1$, $\zeta\in C^\infty(\overline\Omega_c)$ and $\|f\|_{L^2(\Omega_c)}\leq C\|u\|_{H^1(\Omega_c)}$, we obtain
$$u\in H^2(\Omega_1) \quad\text{and}\quad 
\|u\|_{H^2(\Omega_1)}\leq\widehat C\, \|u\|_{H^1(\Omega_c)}. 
$$
Define now 
$$u_n(x,y):=u(x+n,y),\quad n\in\N,\ (x,y)\in\Omega_c.$$
Then, since $\int_{\Omega_c}\nabla u_n\cdot\nabla\varphi\,dx=\int_{\Omega_c}u_n^3\varphi\,dx$ for $\varphi\in H_0^1(\Omega_c)$, by the above we obtain 
$$u_n\in H^2(\Omega_1)\quad\text{and}\quad \|u_n\|_{H^2(\Omega_1)}\leq \widehat C\,\|u_n\|_{H^1(\Omega_c)}.
$$
Using the embedding $H^2(\Omega_1)\hookrightarrow C(\overline{\Omega}_1)$ implies (with $K$
denoting the embedding constant):
$$u_n\in C(\overline\Omega_1)\quad \text{and}\quad \|u_n\|_{C(\overline{\Omega}_1)}\leq K\widehat C \,\|u_n\|_{H^1(\Omega_c)}. $$
Therefore we have $u\in C(\overline\Omega\cap\{x>2.5\})$. Moreover, $\|u_n\|_{H^1(\Omega_c)}=\|u\|_{H^1(\Omega_c+(n,0))}\to 0$ $(n\to\infty)$ since $u\in H^1(\Omega)$, whence
$\|u_n\|_{C(\overline{\Omega}_1)}\to 0$ as $n\to\infty$ follows, implying $u\to 0$ as $x\to\infty$ uniformly in $y$. 
\end{proof}
%
%
%


\bibliographystyle{siam}
\bibliography{bibfile}

\def\cprime{$'$}
\begin{thebibliography}{10}

\bibitem{Ack-Cla-Pa1}
{\sc N.~Ackermann, M.~Clapp, and F.~Pacella}, {\em Self-focusing multibump
  standing waves in expanding waveguides}, Milan J. Math., 79 (2011),
  pp.~221--232.

\bibitem{Ack-Cla-Pa2}
\leavevmode\vrule height 2pt depth -1.6pt width 23pt, {\em Alternating sign
  multibump solutions of nonlinear elliptic equations in expanding tubular
  domains}, Comm. Partial Differential Equations, 38 (2013), pp.~751--779.

\bibitem{Agmon}
{\sc S.~Agmon}, {\em Lectures on Elliptic Boundary Value Problems}, no.~2 in
  Van Nostrand Mathematical Studies, 1965.

\bibitem{BCGP}
{\sc T.~Bartsch, M.~Clapp, M.~Grossi, and F.~Pacella}, {\em Asymptotically
  radial solutions in expanding annular domains}, Math. Ann., 352 (2012),
  pp.~485--515.

\bibitem{BehnkeGoe}
{\sc H.~Behnke and F.~Goerisch}, {\em Inclusions for eigenvalues of selfadjoint
  problems}, in Topics in Validated Computations, J.~Herzberger, ed., Elsevier,
  1994, pp.~277--322.

\bibitem{BehMerPlWie}
{\sc H.~Behnke, U.~Mertins, M.~Plum, and C.~Wieners}, {\em {Eigenvalue
  Inclusions via Domain Decomposition}}, Proc. R. Soc. Lond. A (2000), 456,
  pp.~2717--2730.

\bibitem{BerNir}
{\sc H.~Berestycki and L.~Nirenberg}, {\em {Some Qualitative Properties of
  Solutions of Semilinear Elliptic Equations in Cylindrical Domains}}, in
  {Analysis, et cetera}, P.~H. Rabinowitz, ed., Academic Press, 1990,
  pp.~115--164.

\bibitem{Ref3}
{\sc F.~Blomquist, W.~Hofschuster, and W.~Kr\"amer}, {\em
  C-XSC-Langzahlarithmetiken für reelle und komplexe Intervalle basierend auf
  den Bibliotheken MPFR und MPFI}, Preprint BUW-WRSWT, Universit\"at Wuppertal,
  2011/1.

\bibitem{Breuerbeam}
{\sc B.~Breuer, J.~Horák, P.~J. McKenna, and M.~Plum}, {\em A computer-assisted
  existence and multiplicity proof for travelling waves in a nonlinearly
  supported beam}, Journal of Differential Equations, 224 (2006), pp.~60--97.

\bibitem{CW}
{\sc F.~Catrina and Z.-Q. Wang}, {\em Nonlinear elliptic equations on expanding
  symmetric domains}, J. Differential Equations, 156 (1999), pp.~153--181.

\bibitem{DY}
{\sc E.~N. Dancer and S.~Yan}, {\em Multibump solutions for an elliptic problem
  in expanding domains}, Comm. Partial Differential Equations, 27 (2002),
  pp.~23--55.

\bibitem{GrisvNonsm}
{\sc P.~Grisvard}, {\em Elliptic Problems in Nonsmooth Domains}, no.~24 in
  Monographs and Studiens in Mathematics, Pitman, 1985.

\bibitem{GrisvSing}
\leavevmode\vrule height 2pt depth -1.6pt width 23pt, {\em Singularities in
  Boundary Value Problems}, no.~22 in Research Notes in Applied Mathematics,
  Masson, Springer-Verlag, 1992.

\bibitem{HoangPlWie}
{\sc V.~Hoang, M.~Plum, and C.~Wieners}, {\em A computer-assisted proof for
  photonic band gaps}, Zeitschrift für angewandte Mathematik und Physik, 60
  (2009), pp.~1035--1052.

\bibitem{Kato}
{\sc T.~Kato}, {\em {Perturbation Theory for Linear Operators}}, vol.~132 of
  Grundlehren der mathematischen Wissenschaften, Springer-Verlag, 1976.

\bibitem{Ref2}
{\sc R.~Klatte, U.~Kulisch, A.~Wiethoff, C.~Lawo, and M.~Rauch}, {\em C-XSC:
  A~C++~class library for extended scientific computing}, Springer, 1993.

\bibitem{YYL}
{\sc Y.~Y. Li}, {\em Existence of many positive solutions of semilinear
  elliptic equations on annulus}, J. Differential Equations, 83 (1990),
  pp.~348--367.

\bibitem{McKenPacPlR}
{\sc P.~J. McKenna, F.~Pacella, M.~Plum, and D.~Roth}, {\em {A Uniqueness
  Result for a Semilinear Elliptic Problem: A Computer-assisted Proof}},
  Journal of Differential Equations, 247 (2009), pp.~2140--2162.

\bibitem{MR3203775}
{\sc P.~J. McKenna, F.~Pacella, M.~Plum, and D.~Roth}, {\em A computer-assisted
  uniqueness proof for a semilinear elliptic boundary value problem}, in
  Inequalities and applications 2010, vol.~161 of Internat. Ser. Numer. Math.,
  Birkh\"auser/Springer, Basel, 2012, pp.~31--52.

\bibitem{PlumDMV}
{\sc M.~Plum}, {\em {Existence and Multiplicity Proofs for Semilinear Elliptic
  Boundary Value Problems by Computer Assistance}}, Jahresbericht der DMV, 110
  (2008), pp.~19--54.

\bibitem{Ru99a}
{\sc S.~M. Rump}, {\em {INTLAB - INTerval LABoratory}}, in
  {Developments~in~Reliable Computing}, T.~Csendes, ed., Kluwer Academic
  Publishers, Dordrecht, 1999, pp.~77--104.
\newblock \text{http://www.ti3.tuhh.de/rump/}.

\bibitem{Su-Su}
{\sc C.~Sulem and P.-L. Sulem}, {\em The nonlinear {S}chr\"odinger equation},
  vol.~139 of Applied Mathematical Sciences, Springer-Verlag, New York, 1999.
\newblock Self-focusing and wave collapse.

\bibitem{Su}
{\sc T.~Suzuki}, {\em Positive solutions for semilinear elliptic equations on
  expanding annuli: mountain pass approach}, Funkcial. Ekvac., 39 (1996),
  pp.~143--164.

\bibitem{Ref1}
{\sc C.~Wieners}, {\em A geometric data structure for parallel finite elements
  and the application to multigrid methods with block smoothing}, Comput. Vis.
  Sci., 13 (2010), pp.~161--175.

\end{thebibliography}
\end{document}